\documentclass[12pt,reqno]{amsart}
\usepackage{amssymb}
\usepackage{txfonts}
\usepackage{amsmath}

 \newtheorem{thm}{Theorem}[section]
 
 \newtheorem{lem}[thm]{Lemma}

 \theoremstyle{definition}
 
 \theoremstyle{remark}
 \newtheorem{rem}[thm]{Remark}
 \numberwithin{equation}{section}

\begin{document}

\title[]
{on the extension to mean curvature flow in lower dimension}

\author{Liang Cheng}

\dedicatory{}
\date{}



\address{Liang Cheng, School of Mathematics and Statistics, Huazhong Normal University,
Wuhan, 430079, P.R. CHINA}

\email{math.chengliang@gmail.com}

 \maketitle

\begin{abstract}
In this paper, we prove that if $M^n_t\subset \mathbb{R}^{n+1}$,
$2\leq n\leq 6$, is the $n$-dimensional closed embedded $\mathcal{F}-$stable solution to mean curvature flow with multiplicity one and its mean curvature of $M^n_t$ is uniformly
bounded on $[0,T)$ for $T<\infty$, then the flow can be smoothly extended over time $T$.
\end{abstract}

\section{Introduction}

Let $X_0:M^n\to\mathbb{R}^{n+1}$ be the $n$-dimensional closed embedded submanifold of $\mathbb{R}^{n+1}$ with $M_0=X_0(M^n)$. Consider the mean curvature flow
$$
\frac{\partial X(p,t)}{\partial t}=-H(p,t)\nu(p,t),
$$
satisfying $X(\cdot,t)=X_0(t)$ .
Here $H(p,t)$ and $\nu(p,t)$ are the mean curvature and the outer unit
normal at $X(p,t)$. We denote $A=(h_{ij})$ be the second fundamental form.

A natural question in the study of mean curvature flow is that what are the optimal conditions to guarantee
the existence of a smooth solution to mean curvature flow? First, Huisken proved the following theorem:
\begin{thm}[Huisken \cite{H2}]
Let $M^n_t$ be a smooth closed  solution to mean curvature flow for
$t\in[0,T)$ with $T<\infty$. If the second fundamental form of $M^n_t$ is uniformly
bounded on $[0,T)$, then the flow can be smoothly extended over time $T$.
\end{thm}

So a natural question is that whether the mean curvature flow could be smoothly extended if we only give the assumptions on the mean curvature? In \cite{LS} and \cite{LS1}, Le and Sesum showed that if $M^n_t\subset \mathbb{R}^{n+1}$ is the closed mean convex or Type I solution to mean curvature flow with $||H||_{L^{\alpha}(M\times [0,T))}<\infty$ for $\alpha\geq \frac{n-2}{2}$, then the flow can be smoothly extended over time $T$.
Cooper \cite{Cooper} also proved that if the close mean curvature flow with bounded mean curvature satisfying $|A|^p(T-t)\leq C$ for $p \in (1,2]$, then the flow can be smoothly extended over time $T$. For this topic, one may see \cite{LS} and \cite{LS2} the references therein for more information. It is also conjectured (see \cite{LS2}) that the mean curvature flow for the hypersurfaces of Euclidean space with its dimension less or equal to $7$ can be smoothly extended if the mean curvature stays bounded.

In this paper, we prove that if $M^n_t\subset \mathbb{R}^{n+1}$,
$2\leq n\leq 6$, is the $n$-dimensional closed embedded $\mathcal{F}-$stable solution to mean curvature flow with multiplicity one and its mean curvature of $M_t$ is uniformly
bounded on $[0,T)$ for $T<\infty$, then the flow can be smoothly extended over time $T$.
Recall that, in the beautiful  paper \cite{CM1}, Colding and Minicozzi
studied the entropy stability of self-shrinkers in the mean curvature flow.
The main tool for Colding and Minicozzi's work is the following $\mathcal{F}$-functional
\begin{equation}\label{CM}
F_{y_0,t}(M)=\int_{M}(4\pi t) ^{-\frac{n}{2}}e^{-\frac{|x-y_0|^2}{4t}}d\mu(x), \ t\in (0,+\infty),\ y_0\in\mathbb{R}^{n+m},
\end{equation}
which can be traced back to Huisken¡¯s monotonicity formula \cite{H1}.
A self-shrinker is called $\mathcal{F}$-stable if it is a local minimum for the
$\mathcal{F}$-functional. One is more interested in the $\mathcal{F}$-stable self-shrinkers since
the unstable ones could be perturbed away thus may not represent generic
singularities.
Colding and Minicozzi \cite{CM1} proved the self-shrinkers are the critical points of the $\mathcal{F}$-functional and the spheres,
cylinders and planes are the only $\mathcal{F}$-stable self-shrinkers.

The following theorem is main result of this paper.
\begin{thm}\label{main}
If $M^n_t\subset \mathbb{R}^{n+1}$,
$2\leq n\leq 6$, is the $n$-dimensional closed embedded $\mathcal{F}-$stable solution to mean curvature flow with multiplicity one and its mean curvature of $M^n_t$ is uniformly
bounded on $[0,T)$ for $T<\infty$, then the flow can be smoothly extended over time $T$.
\end{thm}

The organization of the paper is as follows. In section 2, we recall some basic properties for the mean curvature flow.
In section 3, we give the proof of Theorem \ref{main}.

 \section{preliminaries}

In this section, we recall some basic properties for the mean curvature flow.

In \cite{H1}, Huisken introduced his entropy  which becomes one of powerful tools for studying mean curvature flow.
Huisken's entropy enjoys very nice analytic and geometric properties,
including in particular the monotonicity of the entropy.
The Huisken's entropy is defined as the integral of backward heat kernel:
\begin{equation}\label{Huisken}
\theta(M,T,y_0) =\int_{M}[4\pi (T-t)] ^{-\frac{n}{2}}e^{-\frac{|x-y_0|^2}{4(T-t)}}d\mu(x),\ t\in (-\infty,T),\ y_0\in\mathbb{R}^{n}.
\end{equation}
Huisken \cite{H1} proved his entropy (\ref{Huisken}) is monotone non-increasing in $t$ under the mean curvature flow.

Recall the definition of a tangent flow at a point $(y_0,T)$ in space-time of a mean curvature flow. First translate the closed mean curvature
flow $M_t$ in space-time to move $(y_0,T)$ to $(0,0)$ and then take a sequence of parabolic dilations
$(x,t)\to (\lambda_jx,\lambda^2_j t)$ with
$\lambda_j \to \infty$ to get a sequence of mean curvature flow
\begin{equation}\label{rescaled}
M^j_t=\lambda_j(M_{\lambda_j^{-2}t+T}-y_0).
\end{equation}
 Using Huisken's monotonicity formula \cite{H1}, and Brakke's compactness theorem \cite{Bra78} , Ilmanen \cite{I1} and White \cite{W1} show that a subsequence of $M^j_t$ converges weakly to a
limit flow $T_t$ whch is called the tangent flow at $(y_0,T)$.
The tangent flow achieve equality in Huisken's monotonicity formula and, thus, must be a self-shrinker in weak sense.
Precisely,

\begin{thm}\cite{I1} \cite{W1} \label{IW}
For any rescaled mean curvature flow (\ref{rescaled}), any point $(y_0,T)$, there exists a subsequence which we still denote by $\lambda_j$ and a limiting Brakke flow $\nu_t$ such that $\mu^{\lambda_j}_t \rightharpoonup \nu_t$ in the sense of Radon measures for all $t$, where $\mu^{\lambda_j}_t$ is the area measure of $M^{\lambda_j}_t$.  Moreover, $\nu_{-1}$ satisfies
$$ \mathbf{H}_{\infty}(x)+\frac{S(x)^{\perp}\cdot x}{2}=0,\ \ \ \nu_{-1}-a.e. x,$$
where $\mathbf{H}_{\infty}\in L^1_{loc}(\nu_{-1})$ is the mean curvature vector of $\nu_{-1}$ in weak sense, and $S(x)$ denotes the the projection onto the tangent space.
\end{thm}

In \cite{CM1}, Colding and Minicozzi proved the following

\begin{thm}\cite{CM1}\label{CM1}
Let $T_t$ be a tangent flow of a mean curvature flow at a smooth closed embedded
hypersurface in $\mathbb{R}^{n+1}$. If the regular part of $T_{-1}$ is orientale
and $\mathcal{F}-$stable and the singular set has finite $(n-2)$-dimensional Hausdorff measure
and $n\leq 6$, then $T_{-1}$ is either a round sphere or a hyperplane.
\end{thm}

Next we recall White's regularity theorem for the mean curvature flow.
\begin{thm}\cite{W2}\label{W2}
There are numbers $\epsilon=\epsilon(n)>0$ and $C=C(n)<\infty$ with the following property.
If $M_t$ is an $n$-dimensional smooth, proper embedded mean curvature flow in an open subset $U$ of the spacetime
$\mathbb{R}^n\times \mathbb{R}$ and if $\theta(M,T,y_0)$ of $M_t$ are bounded above by $1+\epsilon$, then at
each spacetime point $X=(x,t)$ of $M$, the norm of the second fundamental form of $M_t$ is bounded by $\frac{C}{\delta(X,U)}$,
where $\delta(X,U)$ is the infimum of $||X-Y||$ among all points $Y=(y,s)\in U^c$ with $s\leq t$, where $||\cdot||$ is the parabolic distance.
\end{thm}

Finally, we need the following pseudolocality theorem of the mean curvature flow due to B.L.Chen and Y.Le \cite{CY}.

\begin{thm}\cite{CY}\label{CY}
Let $\bar{M}$ be an $N$-dimensional complete manifold satisfying $\sum\limits^3_{i=0}|\bar{\nabla}\bar{Rm}|\leq c_0^2$
and $inj(\bar{M})\geq i_0>0$. Let $X_0:M\to \bar{M}$ be an $n$-dimensional properly embedded submanifold
with bounded second fundamental form $|A|\leq c_0$ in $\bar{M}$. We assume $M_0=X_0(M)$ is uniform graphic with some
$r>0$. Suppose $X(x,t)$ is a smooth solution to the mean curvature flow on $M\times [0,T_0]$
with $X_0$ as initial data. Then there is $T_1>0$ depending on $c_0,i_0,r$ and the dimension $N$ such that
$$|A|(x,t)\leq 2c_0$$
for all $x\in M$, $0\leq t\leq \min\{T_0,T_1\}$.
\end{thm}

\section{proof of Theorem \ref{main}}
Before presenting the proof of Theorem \ref{main}, we need the following lemma.

\begin{lem}\label{lem}
Let $M_t\subset \mathbb{R}^{n+1}$ be a smooth closed embedded solution to mean curvature flow for
$t\in[0,T)$ with $T<\infty$ is the first singularity time. Moreover, the singularity is multiplicity one. Then the hyperplane cannot arise as a tangent flow to this mean curvature flow.
\end{lem}
\begin{proof}
We argue by contradiction.
Rescale the flow parabolically at $(y_0,T)$ as following
$$M^{\lambda_j}_t=\lambda_j(M_{T+\lambda^{-2}_jt}-y_0),$$
for $\lambda_j\to \infty$. Assume $T_t$ is the tangent flow of $M_t$ at $(y_0,T)$, which is the hyperplane.
By Theorem \ref{IW},there exists a subsequence which we still denote it by $\lambda_j$ and a limiting Brakke flow $\nu_t$
 such that $\mu^{\lambda_j}_{t} \rightharpoonup \nu_t$, where $\mu^{\lambda_j}_t$ is the area measure of $M^{\lambda_j}_t$. Moreover, the singularity is multiplicity one. It follows that $\theta(M^{\lambda_j}_t,0,o)\to \theta(T_t,0,o)=1$ for $t<0$.
By Theorem \ref{W2}, there exists $j_0>0$ such that $|A_j|$ is uniformly bounded for $j\geq j_0$ and $t<0$, where $A_j$ the second fundamental form of $M^{\lambda_j}_t$. So there exists a subsequence of $M^{\lambda_j}_t$ which we still denote it by $M^{\lambda_j}_t$ such that $M^{\lambda_j}_t$ converges to $M^{\infty}_t$ smoothly.

We fix $j$ sufficient large such that $M^{\lambda_j}_{T-T_0}$ satisfies the conditions of Theorem \ref{CY} and $\max\limits_{M^{\lambda_j}_{T-T_0}}|A_j|\leq c_0$. Then we have
\begin{equation}\label{bound}
\max\limits_{M^{\lambda_j}_{t}}|A_j|\leq 2c_0,
\end{equation}
 for $T-T_0\leq t< T$.

 By the Lemma 1.2 in \cite{H1}, we have
$$\max\limits_{M_t}|A|\geq \frac{1}{2(T-t)}$$
if the mean curvature flow blow-up at finite time $T$, where $A$ is the second fundamental form of $M_t$.
Then there exists blow-up sequence $(p_i,t_i)$ with $p_i\to p$ such that
$$|A|(p_i,t_i)\geq \frac{1}{2(T-t_i)}.$$
It follows that
$$|A_j|(p_i,\lambda_j(t_i-T))\geq \frac{1}{2\lambda_j(T-t_i)}.$$
For $i$ large enough, this contradicts to (\ref{bound}).
\end{proof}

Now we give the proof of Theorem \ref{main}.

\textbf{Proof of Theorem \ref{main}.}
We argue by contradiction. Suppose that there exists a mean curvature flow blows up at finite time and satisfying the conditions of Theorem \ref{main}. Rescale the flow parabolically at $(y_0,T)$ as following
$$M^{\lambda_j}_t=\lambda_j(M_{T+\lambda^{-2}_jt}-y_0),$$
for $\lambda_j\to \infty$. Assume $T_t$ is the tangent flow of $M_t$ at $(y_0,T)$. By Theorem \ref{IW},there exists a subsequence which we still denote it by $\lambda_j$ and a limiting Brakke flow $\nu_t$
such that $\mu^{\lambda_j}_{t} \rightharpoonup \nu_t$, where $\mu^{\lambda_j}_t$ is the area measure of $M^{\lambda_j}_t$.
 Moreover, $\nu_{-1}$ satisfies
$$ \mathbf{H}_{\infty}(x)+\frac{S(x)^{\perp}\cdot x}{2}=0,\ \ \ \nu_{-1}-a.e. x,$$
where $\mathbf{H}_{\infty}\in L^1_{loc}(\nu_{-1})$ is the mean curvature vector of $\nu_{-1}$ in weak sense, and $S(x)$ denotes the the projection onto the tangent space. Since mean curvature of $M_t$ is uniformly bounded, $\mathbf{H}_{\infty} \equiv 0\ a.e.$  It follows that the tangent flow $T_t$ is the minimal cone and the conditions of Theorem \ref{CM1} are satisfied. By Theorem \ref{CM1}, we conclude that $T_{-1}$ is either a round sphere or a hyperplane. This contradicts to Lemma \ref{lem} and $\mathbf{H}_{\infty} \equiv 0$.
$\Box$

\begin{rem} We remark that the condition $|H|\leq C$ on $[0,T)$ in Theorem \ref{main} can be replaced by the weak condition
$||H||_{L^\alpha(M_t\times [0,T))}<\infty$ for $\alpha\geq \frac{n+2}{2}$, where $H(\cdot,t)$ is the mean curvature of $M_t$. The reason is as follows. We still rescale the flow parabolically at $(y_0,T)$ as following
$$M^{\lambda_j}_t=\lambda_j(M_{T+\lambda^{-2}_jt}-y_0),$$
for $\lambda_j\to \infty$ and denote $H_j(\cdot,t)$ be the mean curvature of $M^{\lambda_j}_t$.  Assume $T_t$ is the tangent flow of $M_t$ at $(y_0,T)$ and
$H_{\infty}(\cdot,t)$ is the mean curvature of $T_t$.
 We calculate
\begin{eqnarray*}
\int^{-1}_{-2}\int_{B(o,R)\cap T_t}
|H_{\infty}|^{\frac{n+2}{2}} d\nu_t dt &\leq&
\lim\limits_{j\to\infty} \int^{-1}_{-2}\int_{B(o,R)\cap M^j_t} |H_{j}|^{\frac{n+2}{2}} d\mu^j_t dt\\
& =& \lim\limits_{j\to\infty} \int^{T-(\lambda_j)^{-1}}_{T-2(\lambda_j)^{-1}}
\int_{B(y_0,\lambda_j^{-\frac{1}{2}})\cap M_t} |H|^{\frac{n+2}{2}} d\mu_t dt\\
&\leq& \lim\limits_{j\to\infty} \int^{T-(\lambda_j)^{-1}}_{T-2(\lambda_j)^{-1}}\int_{M_t} |H|^{\frac{n+2}{2}} d\mu_t dt\\
& =& 0.
\end{eqnarray*}
The last equality holds, since $\int^T_0\int_{M_t} |H|^{\frac{n+2}{2}}
d\mu dt<\infty$. Then we conclude that $H_{\infty} \equiv 0$ a.e. Then we still can argue by the same way as the the proof of Theorem \ref{main}.
\end{rem}

\thanks{\textbf{Acknowledgement}: The author would like to thank for Dr.Anqiang Zhu, Dr.hongbin Qiu and Prof.Gaofeng Zhen for useful talking.}

\end{document}